\documentclass{amsart}
\usepackage{euler, hyperref}
\usepackage[all]{xy}

\theoremstyle{plain}
\newtheorem{theorem}{Theorem}
\newtheorem{proposition}[theorem]{Proposition}
\newtheorem*{notations*}{Notations}

\theoremstyle{definition}
\newtheorem{remark}[theorem]{Remark}

\newcommand{\f}{\varphi}

\newcommand{\CC}{\mathbb C}
\newcommand{\PP}{\mathbb P}
\newcommand{\ZZ}{\mathbb Z}

\newcommand{\C}{{\mathcal C}}
\newcommand{\F}{{\mathcal F}}
\newcommand{\G}{{\mathcal G}}
\newcommand{\I}{{\mathcal I}}
\newcommand{\J}{{\mathcal J}}
\def\O{\mathcal O}

\newcommand{\Ker}{{\mathcal Ker}}
\newcommand{\Coker}{{\mathcal Coker}}
\newcommand{\Image}{{\mathcal Im}}

\newcommand{\Aut}{\operatorname{Aut}}
\newcommand{\Ext}{\operatorname{Ext}}
\newcommand{\Hom}{\operatorname{Hom}}
\newcommand{\D}{{\scriptscriptstyle \operatorname{D}}}
\newcommand{\E}{\operatorname{E}}
\newcommand{\GL}{\operatorname{GL}}
\newcommand{\h}{\operatorname{h}}
\def\H{\operatorname{H}}
\newcommand{\Hilb}{\operatorname{Hilb}_{{\mathbb P}^2}}
\newcommand{\M}{\operatorname{M}_{{\mathbb P}^2}}
\newcommand{\N}{\operatorname{N}}
\newcommand{\st}{{\scriptstyle \operatorname{s}}}
\newcommand{\sst}{{\scriptstyle \operatorname{ss}}}

\newcommand{\tensor}{\otimes}
\newcommand{\isom}{\simeq}
\newcommand{\lra}{\longrightarrow}

\newcommand{\ds}{\displaystyle}
\newcommand{\ba}{\begin{array}}
\newcommand{\ea}{\end{array}}

\begin{document}

\title[Relative Hilbert schemes and moduli spaces of torsion plane sheaves]
{Relative Hilbert schemes and moduli spaces of torsion plane sheaves}

\author{Mario Maican}
\address{Institute of Mathematics of the Romanian Academy,
Calea Grivi\c{t}ei 21, Bucure\c{s}ti 010702, Romania}
\email{mario.maican@imar.ro}


\maketitle

Let $\M(r, \chi)$ denote the moduli space of Gieseker semi-stable sheaves $\F$ on the complex projective plane $\PP^2$
having Hilbert polynomial $P_{\F}(m) = rm + \chi$, where $r$ and $\chi$ are fixed integers, $r > 0$.
Let $\Hilb(l, d)$ denote the flag Hilbert scheme of pairs  $(Z, C)$, where $C \subset \PP^2$ is a curve
of degree $d$ and $Z \subset C$ is a zero-dimensional subscheme of length $l$. Here $l$ and $d$ are
fixed positive integers.
In this note we will exhibit certain relations between $\Hilb(l, d)$ and $\M(r, \chi)$.
This will allow us to prove that the moduli spaces $\M(r, \pm 1)$ are rational.
Le~Potier \cite{lepotier} showed by a different method that $\M(r, \chi)$ are rational if $\chi = \pm 1$, $\pm 2$.
At Proposition \ref{proposition_9} we will prove that some of the spaces $\M(r, \chi)$ are unirational.

\begin{notations*} \hfill

\begin{tabular}{r c l}
$p(\F)$ & $=$ & $\chi(\F)/r$, the slope of $\F$; \\
$\F^\D$ & $=$ & ${\mathcal Ext}^1(\F, \omega_{\PP^2})$, the dual of a one-dimensional sheaf $\F$ on $\PP^2$; \\
$[\F]$ & $=$ & the stable-equivalence class of $\F$; \\
$\CC_P$ & $=$ & the structure sheaf of a closed point $P \in \PP^2$; \\
$\Hilb(l)$ & $=$ & the Hilbert scheme of zero-dimensional subschemes of $\PP^2$ of length $l$.
\end{tabular}
\end{notations*}

\begin{proposition}
\label{proposition_1}
Assume that $d \ge 2$.
Let $n$ be an integer such that $1 \le n \le d-1$.
Let $\J_Z \subset \O_C$ be the ideal sheaf of $Z$ in $C$.
The sheaf $\J_Z$ is semi-stable in either of the following cases:
\begin{enumerate}
\item[(i)] ${\ds l \le \frac{d}{2} }$; or
\item[(ii)] ${\ds l \le \frac{d}{2} + \frac{dn(n+1)}{2(d-1)}}$ and $Z$ is not a subscheme of a curve of degree $n$.
\end{enumerate}
If the inequalities are strict, then $\J_Z$ is stable.
Moreover, the bound in (i) is sharp.
\end{proposition}

\begin{proof}
The structure sheaf $\O_C$ of $C$ has no zero-dimensional torsion, so the same is true of $\J_Z$.
Let $\J \subset \J_Z$ be a subsheaf of multiplicity at most $d - 1$.
According to \cite[Lemma 6.7]{pacific}, there is a curve $S \subset C$ of degree $s \le d-1$ such that
its ideal sheaf $\J_S \subset \O_C$ contains $\J$ and $\J_S/\J$ is supported on finitely many points.
Thus,
\begin{align*}
P_{\J}(m) & = P_{\J_S}(m) - \h^0(\J_S/\J) \\
& = P_{\O_C}(m) - P_{\O_S}(m) - \h^0(\J_S/\J) \\
& = dm - \frac{d(d-3)}{2} - \left( sm - \frac{s(s-3)}{2} \right) - \h^0(\J_S/\J),
\end{align*}
hence
\[
p(\J) = - \frac{d+s}{2} + \frac{3}{2} - \frac{\h^0(\J_S/\J)}{d-s}. \qquad \text{We know that} \quad
p(\J_Z) = -\frac{d}{2} + \frac{3}{2} - \frac{l}{d}.
\]
It follows that $p(\J) \le p(\J_Z)$ if and only if
\[
\frac{l}{d} \le \frac{s}{2} + \frac{\h^0(\J_S/\J)}{d-s}.
\]
The right-hand-side is at least $1/2$, which proves (i).
Assume now that $Z$ is not a subscheme of a curve of degree $n$.
Let $\I_Z$, $\I_S$, $\I \subset \O$ be the preimages of $\J_Z$, $\J_S$, $\J$ under the map $\O \to \O_C$.
We claim that, if $s \le n$, then
\[
\h^0(\J_S/\J) \ge \dim \Gamma (\I_S(n)) = {n - s + 2 \choose 2}.
\]
Indeed, we have an exact sequence of linear maps
\[
0 \lra \Gamma(\I(n)) \lra \Gamma(\I_S(n)) \lra \Gamma((\I_S/\I)(n)) = \Gamma((\J_S/\J)(n)) = \Gamma(\J_S/\J).
\]
If the vector space on the right-hand-side had smaller dimension than $\Gamma(\I_S(n))$,
then we would find a non-zero element of $\Gamma(\I(n))$, which is contained in $\Gamma(\I_Z(n))$.
This would contradict our hypothesis on $Z$, proving the claim.
It follows that the inequalities
\begin{align*}
\frac{l}{d} & \le \frac{s}{2} + \frac{1}{d - s} {n - s + 2 \choose 2} & & \text{for} \quad 1 \le s \le n, \\
\frac{l}{d} & \le \frac{s}{2} & & \text{for} \quad n+1 \le s \le d-1
\end{align*}
guarantee the semi-stability of $\J_Z$. For $1 \le s \le n \le d-1$ we have the inequalities
\[
\frac{1}{2} + \frac{1}{d - 1} {n + 1 \choose 2} \le \frac{s}{2} + \frac{1}{d - s} {n - s + 2 \choose 2}.
\]
Moreover, if $n \le d - 2$, then
\[
\frac{1}{2} + \frac{1}{d - 1} {n + 1 \choose 2} \le \frac{n+1}{2}.
\]
This proves (ii).
The bound in (i) is sharp, meaning that, whenever $l > d/2$, we can find $(Z, C)$
such that $\J_Z$ is unstable. Take $C=L \cup C'$ for a line $L$ and a curve $C'$,
take $Z$ a subscheme of $L$. The ideal of $L$ in $C$ is a destabilising subsheaf of $\J_Z$.
\end{proof}

\noindent
Under hypothesis (i) of Proposition \ref{proposition_1}, we have a map
\[
\eta \colon \Hilb(l,d) \lra \M \big( d, - \frac{d(d-3)}{2} - l \big), \qquad
\eta(Z,C) = [\J_Z].
\]
This map is a morphism because it is associated to the flat family $\{ \J_Z \}$ over $\Hilb(l, d)$.
The image of $\eta$, denoted by $\H(l ,d)$, is a closed subset of the moduli space;
we equip it with the induced reduced structure.
Under hypothesis (ii), we can define $\eta$ on an open subset of the Hilbert scheme.

\begin{proposition}
\label{proposition_2}
For all $d \ge 3$ the map $\eta \colon \Hilb(1,d) \to \H(1,d)$ is an isomorphism.
\end{proposition}

\begin{proof}
Clearly, $\eta$ is bijective.
To prove that $\eta^{-1}$ is a morphism we will construct local inverse morphisms as at \cite[Theorem 3.1.6]{dedicata}.
Given $\F = \J_P \subset \O_C$ we need to construct the pair $(P, C) \in \Hilb(1,d)$ from the first sheet
$\E^1(\F)$ of the Beilinson spectral sequence converging to $\F$, by performing algebraic operations.
For technical reasons we will, instead, work with the Beilinson spectral sequence of the dual sheaf
$\G = \F^\D$, which, according to \cite{rendiconti}, gives a point in $\M(d, d(d-3)/2 +1)$.
Dualising the resolution
\[
0 \lra \O(-d) \oplus \O(-2) \lra 2\O(-1) \lra \F \lra 0
\]
we get the exact sequence
\[
0 \lra 2\O(-2) \lra \O(-1) \oplus \O(d-3) \lra \G \lra 0,
\]
hence the extension
\[
0 \lra \O_C(d-3) \lra \G \lra \CC_P \lra 0.
\]
The relevant part of $\E^1(\G)$ is
\[
\xymatrix
{
\H^1(\G(-1)) \tensor \O(-2) \ar[r]^-{\f_1} & \H^1(\G \tensor \Omega^1(1)) \tensor \O(-1) \ar[r]^-{\f_2}
& \H^1(\G) \tensor \O \\
\H^0(\G(-1)) \tensor \O(-2) \ar[r]^-{\f_3} & \H^0(\G \tensor \Omega^1(1)) \tensor \O(-1) \ar[r]^-{\f_4}
& \H^0(\G) \tensor \O
}
\]
We have $\h^1(\G(-1))=2$, $\h^1(\G \tensor \Omega^1(1))=1$, $\h^1(\G) =0$.
As $\G$ maps surjectively onto $\Coker(\f_1)$, we have the isomorphisms
$\Coker(\f_1) \isom \CC_Q$ for some $Q \in \PP^2$ and $\Ker(\f_1) \isom \O(-3)$
(the other possibility, that $\Coker(\f_1) \isom \O_L(-1)$ for a line $L \subset \PP^2$,
violates the semi-stability of $\G$).
The exact sequence (2.2.5) from \cite{dedicata}, takes the form
\[
0 \lra \O(-3) \stackrel{\f_5}{\lra} \Coker(\f_4) \lra \G \lra \CC_Q \lra 0.
\] 
Put $\G'=\Coker(\f_5)$. From the exact sequence (2.2.4) in \cite{dedicata}, we get the resolution
\[
0 \to \H^0(\G(-1)) \tensor \O(-2) \to
\O(-3) \oplus \H^0(\G \tensor \Omega^1(1)) \tensor \O(-1) \stackrel{\f'}{\to} \H^0(\G) \tensor \O \to \G' \to 0.
\]
Thus $\H^0(\G')=\H^0(\G)$.
Since $\G'$ is generated by its global sections, we see that $\G' = \O_C(d-3)$.
It is also clear that $P = Q$. We conclude that
\[
(\CC_P, \O_C)=(\Coker(\f_1), \ \Coker(\f')(-d+3)),
\]
so the pair $(P, C)$ depends algebraically on $\E^1(\G)$.
\end{proof}

\noindent
The case $d = 3$ of Proposition \ref{proposition_2} is known from \cite{lepotier}
(in this case $\H(1, 3)$ is the entire moduli space).
The cases $d=4, 5$ were dealt with at \cite[Theorem 3.2.1]{dedicata}, respectively, \cite[Proposition 3.2.5]{illinois}.

\begin{proposition}
\label{proposition_3}
For all $d \ge 5$ the map $\eta \colon \Hilb(2,d) \to \H(2,d)$ is an isomorphism.
\end{proposition}

\begin{proof}
Clearly, $\eta$ is bijective. As at Proposition \ref{proposition_2},
given $\F= \J_Z \subset \O_C$, we need to construct $(Z,C) \in \Hilb(2,d)$ starting from $\E^1(\F)$.
For technical reasons we will work, instead, with the Beilinson spectral sequence of the sheaf $\G= \F^\D(-1)$, which,
according to \cite{rendiconti}, gives a point in $\M(d,d(d-5)/2+2)$.
Dualising the resolution
\[
0 \lra \O(-d) \oplus \O(-3) \lra \O(-2) \oplus \O(-1) \lra \F \lra 0
\]
we get the exact sequence
\[
0 \lra \O(-3) \oplus \O(-2) \lra \O(-1) \oplus \O(d-4) \lra \G \lra 0,
\]
hence the extension
\[
0 \lra \O_C(d-4) \lra \G \lra \O_Z \lra 0.
\]
Tableau (2.2.3) \cite{dedicata} for $\G$ has the form
\[
\xymatrix
{
4\O(-2) \ar[r]^-{\f_1} & 4\O(-1) \ar[r]^-{\f_2} & \O \\
\H^0(\G(-1)) \tensor \O(-2) \ar[r]^-{\f_3} & \H^0(\G \tensor \Omega^1(1)) \tensor \O(-1) \ar[r]^-{\f_4}
& \H^0(\G) \tensor \O
}
\]
Write $\C = \Ker(\f_2)/\Image(\f_1)$.
Note that $\Ker(\f_2) \isom \Omega^1 \oplus \O(-1)$ because $\f_2$ is surjective.
The Euler sequence on $\PP^2$ reads
\[
0 \lra \O(-3) \lra 3\O(-2) \stackrel{\pi}{\lra} \Omega^1 \lra 0.
\]
Clearly, the corestriction $4\O(-2) \to \Omega^1 \oplus \O(-1)$ of $\f_1$ factors through the map
\[
(\pi, id) \colon 3\O(-2) \oplus \O(-1) \lra \Omega^1 \oplus \O(-1).
\]
We obtain an exact sequence of the form
\[
0 \lra \Ker(\f_1) \lra \O(-3) \oplus 4\O(-2) \stackrel{\psi}{\lra} 3\O(-2) \oplus \O(-1) \lra \C \lra 0.
\]
As at \cite[Proposition 2.1.4]{illinois}, it can be shown that $\psi_{12}$ has maximal rank.
Canceling $3\O(-2)$, we arrive at the exact sequence
\[
0 \lra \Ker(\f_1) \lra \O(-3) \oplus \O(-2) \lra \O(-1) \lra \C \lra 0.
\]
Note that $\C$ cannot be isomorphic to the twisted structure sheaf of a line or of a conic curve,
otherwise $\C$ would destabilise $\G$.
We deduce that $\C \isom \O_Y$ for a zero-dimensional subscheme $Y \subset \PP^2$ of length $2$,
and that $\Ker(\f_1) \isom \O(-4)$.
The exact sequence (2.2.5) from \cite{dedicata} takes the form
\[
0 \lra \O(-4) \stackrel{\f_5}{\lra} \Coker(\f_4) \lra \G \lra \O_Y \lra 0.
\] 
Denote $\G'=\Coker(\f_5)$. From the exact sequence (2.2.4) in \cite{dedicata} we get the resolution
\[
0 \to \H^0(\G(-1)) \tensor \O(-2) \to
\O(-4) \oplus \H^0(\G \tensor \Omega^1(1)) \tensor \O(-1) \stackrel{\f'}{\to} \H^0(\G) \tensor \O \to \G' \to 0.
\]
Thus $\H^0(\G')=\H^0(\G)$.
Since $\G'$ is generated by its global sections, we see that $\G' = \O_C(d-4)$.
It is also clear that $Z=Y$. We conclude that
\[
(\O_Z, \O_C)=(\Coker(\f_1), \ \Coker(\f')(-d+4)),
\]
so the pair $(Z,C)$ depends algebraically on $\E^1(\G)$.
\end{proof}

\noindent
The cases $d=4, 5$ of Proposition \ref{proposition_3} were dealt with at
\cite[Section 5.2]{dedicata}, respectively, \cite[Proposition 2.2.5]{illinois}.
Note that, when $d = 4$, $\eta$ gives an isomorphism on the preimage of the set of stable points.
The semi-stable but not stable points in $\H(2, 4)$ are of the form $[\O_L(-2) \oplus \O_{C'}(-1)]$,
where $L \subset \PP^2$ is a line and $C' \subset \PP^2$ is a cubic curve.
The fibre of $\eta$ over such a point is isomorphic to the Hilbert scheme of two points on $L$,
that is, to $\PP^1 \times \PP^1$.

Let $n \ge 2$ be an integer, put $l = n(n+1)/2$ and denote by
$\Hilb^0(l, d)$ the open subset of $\Hilb(l, d)$ given by the condition that $Z$
be not a subscheme of a curve of degree $n-1$.
Let $\Hilb^0(l)$ the open subset of $\Hilb(l)$ given by the same condition.
Here we assume that $n \le d$, which implies condition (ii) from Proposition \ref{proposition_1}.
Thus, $\eta$ is defined on $\Hilb^0(l, d)$; we denote its image by $\H_0(l, d)$.

\begin{proposition}
\label{proposition_4}
Assume that $n+2 \le d$.
Then $\eta \colon \Hilb^0(l, d) \to \H_0(l, d)$ is an isomorphism.
\end{proposition}

\begin{proof}
According to \cite[Propositions 4.5 and 4.6]{modules-alternatives}, the ideal sheaves
$\I_Z \subset \O_{\PP^2}$ of zero-dimensional subschemes $Z \subset \PP^2$
of length $l$ that are not contained in curves of degree $n-1$
are precisely the sheaves having a resolution of the form
\[
0 \lra n\O(-n-1) \stackrel{\psi}{\lra} (n+1)\O(-n) \lra \I_Z \lra 0,
\]
where the maximal minors of $\psi$ have no common factor.
It follows that the sheaves $\J_Z$ giving points in $\H_0(l, d)$ are precisely the sheaves having
a resolution of the form
\[
0 \lra \O(-d) \oplus n\O(-n-1) \stackrel{\f}{\lra} (n+1)\O(-n) \lra \J_Z \lra 0,
\]
where the maximal minors of $\f_{12}$ have no common factor.
From this it immediately becomes clear that $\eta$ is injective.
Indeed, any isomorphism $\J_Z \isom \J_Y$ must fit into a commutative diagram
having two horizontal rows that are resolutions as above and such that the vertical arrows are isomorphisms.
This forces $Z = Y$.

The dual sheaf $\G = \J_Z^\D(-n+1)$ gives a point in $\M(d, d(d-2n-1)/2+ l)$.
By duality, we have a resolution of the form
\[
0 \lra (n+1)\O(-2) \lra n\O(-1) \oplus \O(d-n-2) \lra \G \lra 0,
\]
hence we have an extension of the form
\[
0 \lra \O_C(d-n-2) \lra \G \lra \O_Z \lra 0.
\]
Tableau (2.2.3) \cite{dedicata} for $\G$ has the form
\[
\xymatrix
{
(n+1) \O(-2) \ar[r]^-{\f_1} & n\O(-1) \ar[r]^-{\f_2} & 0 \\
\H^0(\G(-1)) \tensor \O(-2) \ar[r]^-{\f_3} & \H^0(\G \tensor \Omega^1(1)) \tensor \O(-1) \ar[r]^-{\f_4}
& \H^0(\G) \tensor \O
}
\]
Clearly, $\Coker(\f_1)$ has support of dimension $0$ or $1$, being a quotient sheaf of $\G$.
Thus, at least one of the maximal minors of $\f_1$ is non-zero.
It follows that $\Ker(\f_1) \isom \O(-k -2)$, where $n - k$ is the degree of the greatest common
divisor of the maximal minors of $\f_1$.
The exact sequence (2.2.5) from \cite{dedicata} takes the form
\[
0 \lra \O(-k-2) \stackrel{\f_5}{\lra} \Coker(\f_4) \lra \G \lra \Coker(\f_1) \lra 0.
\]
Denote $\G'=\Coker(\f_5)$. From the exact sequence (2.2.4) in \cite{dedicata} we get the resolution
\begin{multline*}
0 \lra \H^0(\G(-1)) \tensor \O(-2) \lra \\
\O(-k-2) \oplus \H^0(\G \tensor \Omega^1(1)) \tensor \O(-1) \stackrel{\f'}{\lra} \H^0(\G) \tensor \O \lra \G' \lra 0.
\end{multline*}
Thus $\H^0(\G')=\H^0(\G)$.
Since $\G'$ is generated by its global sections, we see that $\G' = \O_C(d-n-2)$.
It follows that $\Coker(\f_1) \isom \O_Z$. We conclude that
\[
(\O_Z,\O_C)=(\Coker(\f_1), \ \Coker(\f')(-d+n+2)),
\]
so the pair $(Z,C)$ depends algebraically on $\E^1(\G)$.
\end{proof}

\noindent
The cases $(l, d) = (3, 4)$, $(3, 5)$ of Proposition \ref{proposition_4} were dealt with at
\cite[Proposition 3.3.2]{dedicata}, respectively, \cite[Proposition 2.3.4]{illinois}.

\begin{remark}
\label{remark_5}
For $d \ge n$ the canonical morphism $\Hilb^0(l, d) \to \Hilb^0(l)$ of forgetting the curve is a projective bundle
with fibre of dimension
\[
\frac{(d+1)(d+2)}{2} - \frac{n(n + 1)}{2} -1.
\]
In other words, if a zero-dimensional scheme $Z \subset \PP^2$ of length $n(n+1)/2$ imposes the maximal number of conditions
on curves of degree $n-1$, then it imposes the maximal number of conditions on curves of degree $d$ for all $d \ge n$.
This is a well-known fact, which follows from the resolution of $\I_Z$ given at Proposition \ref{proposition_4}:
\[
\h^0(\I_Z(d)) = (n+1) {d - n + 2 \choose 2} - n {d - n + 1 \choose 2} = \frac{(d+1)(d+2)}{2} - \frac{n(n + 1)}{2}.
\]
Since $\Hilb(l)$ is a rational variety, we deduce that $\Hilb(l, d)$ is rational, too.
\end{remark}

\noindent
An important case of Proposition \ref{proposition_4} occurs when $d = n + 2$. In this case $\H_0(l, d)$ is an open subset
of $\M(n+2, - n^2 - n + 1) \isom \M(n+2, n+1)$.

\begin{proposition}
\label{proposition_6}
Let $n \ge 2$ be an integer, let $l = (n-1)n/2$. Then $\Hilb^0(l, n+1)$, which is a projective bundle over $\Hilb^0(l)$,
is isomorphic to an open subset of $\M(n+1, n)$. Thus, $\M(n+1, n)$ is rational.
\end{proposition}

\noindent
This proposition is also an immediate consequence of \cite[Propositions 7.6 and 7.7]{pacific}.
The algebraic group
\[
G = (\GL(n-1, \CC) \times \GL(n, \CC))/\CC^*
\]
acts by conjugation on the vector space
of $n \times (n-1)$-matrices with entries linear forms in three variables.
There is a geometric quotient, denoted by $\N(3, n - 1, n)$, of the set of semi-stable matrices modulo $G$.
Let $\N_0(3, n - 1, n)$ denote the open subset given by the condition that the maximal minors of the matrix
have no common factor. According to \cite[Propositions 4.5 and 4.6]{modules-alternatives},
\[
\N_0(3, n - 1, n) \isom \Hilb^0((n-1)n/2).
\]
In \cite{pacific} we proved that the open subset of $\M(n+1, n)$ given by the condition $\H^0(\F(-1)) = 0$
is isomorphic to an open subset of a certain projective bundle with base $\N(3, n - 1, n)$.
If we restrict the base, we obtain an open subset of $\M(n+1, n)$ isomorphic
to a projective bundle over $\N_0(3, n - 1, n)$. This bundle is $\Hilb^0(l, n+1)$.

We will next attempt to address the question of the rationality of $\M(r, \chi)$ in general.
Let $n$ and $r$ be integers such that $2 \le r \le n$.
We consider the open subset $M_0 \subset \M(n+r, n)$
of sheaves $\F$ that have smooth schematic support and that satisfy the conditions
\[
\H^0(\F(-1)) = 0, \qquad \H^1(\F \tensor \Omega^1(1))=0, \qquad \H^1(\F) = 0.
\]
Consider the vector space
\[
W = \Hom(r\O(-2) \oplus (n-r) \O(-1), n\O)
\]
and the algebraic group
\[
G_W = (\Aut(r\O(-2) \oplus (n-r) \O(-1)) \times \Aut(n\O))/\CC^*
\]
acting on $W$ by conjugation.
Consider the open $G_W$-invariant subset $W_0 \subset W$
of injective morphisms whose determinant gives a smooth curve in $\PP^2$.
The cokernel of any morphism $\f \in W_0$ is stable because it is a line bundle on a smooth curve.
Thus, we have a morphism $\rho \colon W_0 \to M_0$, given by $\rho(\f) = [\Coker(\f)]$.

\begin{remark}
\label{remark_7}
The map $\rho \colon W_0 \to M_0$ is a geometric quotient map modulo $G_W$.
Indeed, for any sheaf $\F$ giving a point in $M_0$, the Beilinson Spectral Sequence with $\E^1$-term
\[
\E^1_{ij} = \H^j(\F \tensor \Omega^{-i}(-i)) \tensor \O(i)
\]
converges to $\F$ and leads to a resolution of the form
\[
0 \lra r \O(-2) \oplus (n-r) \O(-1) \stackrel{\f}{\lra} n\O \lra \F \lra 0.
\]
Thus $\rho$ is surjective. Clearly, its fibres are the $G_W$-orbits.
Note that $M_0$ is smooth, being contained in the stable locus of $\M(n+r, n)$.
We can now apply \cite[Theorem 4.2]{popov} (which only requires the hypothesis that $M_0$ be normal)
to conclude that $M_0$ is the geometric quotient of $W_0$ modulo $G_W$.
\end{remark}

Let $l = (n+r)(n+r-1)/2 - n$ and let $\Hilb^0(l) \subset \Hilb(l)$ be the open subset of subschemes $Z \subset \PP^2$
whose ideal sheaf $\I_Z \subset \O_{\PP^2}$ satisfies the cohomological conditions
\[
\H^0(\I_Z(n+r-3)) = 0, \quad \H^1(\I_Z(n+r-1) \tensor \Omega^1) = 0, \quad \H^1(\I_Z(n+r-2)) = 0.
\]
Consider the vector space
\[
U = \Hom((r-1)\O(-2) \oplus (n-r) \O(-1), n\O)
\]
and the algebraic group
\[
G_U = (\Aut((r-1)\O(-2) \oplus (n-r) \O(-1)) \times \Aut(n\O))/\CC^*
\]
acting on $U$ by conjugation. 
Consider the open $G_U$-invariant subset $U_0 \subset U$
of injective morphisms whose maximal minors have no common factor.
The cokernel of any morphism $\psi \in U_0$ is of the form $\I_Z(n+r-2)$ for some $Z \in \Hilb^0(l)$.
Mapping $\psi$ to $Z$ gives us a morphism $\zeta \colon U_0 \to \Hilb^0(l)$.

\begin{proposition}
\label{proposition_8}
The morphism $\zeta$ is a geometric quotient map modulo $G_U$.
\end{proposition}

\begin{proof}
We apply the Beilinson Spectral Sequence from Remark \ref{remark_7} to $\I_Z(n+r-2)$.
By hypothesis, the terms $\E^1_{-2, 0}$,
$\E^1_{-1, 1}$ and $\E^1_{01}$ vanish. By Serre Duality we have the isomorphisms
\[
\H^2(\I_Z(n+r-2)) \isom \Hom(\I_Z(n+r-2), \omega_{\PP^2})^* \isom \Hom(\I_Z, \O(-n-r-1)).
\]
The group on the right vanishes, as can be seen from the exact sequence
\[
0 = \Hom(\O, \O(-n-r-1)) \lra \Hom(\I_Z, \O(-n-r-1)) \lra \Ext^1(\O_Z, \O(-n-r-1)).
\]
The group on the right vanishes. Indeed, we notice that for a closed point $P$ in $\PP^2$
\[
\Ext^1(\CC_P, \O(-n-r-1))=0
\]
and then we apply induction on the length of $Z$.
Thus, $\E^1_{02} = 0$ and, analogously, $\E^1_{-2, 2}=0$.
It becomes clear now that $\E^1_{-1,2} = \E^{\infty}_{-1,2} = 0$.
We have
\begin{align*}
\ba{lclcl}
\h^1(\I_Z(n+r-3)) & = & -\chi(\I_Z(n+r-3)) & = & r-1, \\
\h^0(\I_Z(n+r-1) \tensor \Omega^1) & = & \phantom{-}\chi(\I_Z(n+r-1) \tensor \Omega^1) & = & n-r, \\
\h^0(\I_Z(n+r-2)) & = & \phantom{-}\chi(\I_Z(n+r-2)) & = & n.
\ea
\end{align*}
From the term $\E^3 = \E^{\infty}$ of the spectral sequence we obtain a resolution of the form
\[
0 \lra (r-1)\O(-2) \oplus (n-r) \O(-1) \stackrel{\psi}{\lra} n\O \lra \I_Z(n+r-2) \lra 0.
\]
Clearly, $\psi \in U_0$, which proves that $\zeta$ is surjective.
Its fibres are precisely the $G_U$-orbits.
The above construction of $\psi$ starting from $\E^1$ gives local inverse morphisms to $\zeta$.
We conclude that $\zeta$ is a geometric quotient map for the action of $G_U$.
\end{proof}

Consider the open subset $H_0 \subset \Hilb(l, n + r)$ of pairs $(Z, C)$ such that $C$ is smooth and $Z \in \Hilb^0(l)$.
Denote by $\J_Z$ the ideal sheaf of $Z$ in $C$. Consider $\psi \in U_0$
such that $\I_Z(n+r-2) \isom \Coker(\psi)$. We can find $\f \in W$ such that $\J_Z(n+r-2) \isom \Coker(\f)$
and the restriction of $\f$ to a direct summand of $\O(-2)$ is $\psi$.
As $\det(\f)$ defines the smooth curve $C$, we see that $\f \in W_0$.
We have constructed a morphism
\[
\eta \colon H_0 \to M_0, \qquad \eta(Z,C) = [\J_Z(n+r-2)].
\]

Let $q$, $m$, $n$ be positive integers and consider the vector space $K = K(q, m, n)$ of $n \times m$-matrices with
entries in $\CC^q$. The elements $K$ are called \emph{Kronecker modules}. The algebraic group
\[
G = (\GL(m, \CC) \times \GL(n, \CC))/\CC^*
\]
acts on $K$ by conjugation.
By Geometric Invariant Theory, the set of semi-stable points $K^\sst$, if non-empty, admits a good quotient
modulo $G$, denoted $\N(q, m, n)$, which is a projective variety of dimension $qmn- m^2 - n^2 +1$.
According to \cite{drezet_reine}, $K^\sst \neq \emptyset$ and $\dim \N(q, m , n) > 0$ if and only if
\[
x < \frac{m}{n} < \frac{1}{x},
\]
where $x$ is the smaller solution to the equation $x^2 -qx + 1 = 0$.
In this case, the set $K^\st$ of stable points is also non-empty.

Assume now that $r < n$. Let $U^\st \subset U$ be the open $G_U$-invariant subset of morphisms $\psi$ such that
\[
\psi_{12} \in \Hom((n-r) \O(-1), n\O) = K(3, n-r, n)
\]
is stable as a Kronecker module.
Let $x$ be the smaller solution to the equation $x^2 - 3x + 1 = 0$.
As mentioned above, $K(3, n-r, n)^\st$ is non-empty if
\[
x < \frac{n-r}{n} < \frac{1}{x}.
\]
This corresponds to the case when $\dim \N(3, n-r, n) > 0$.
We examine separately the case when $\dim \N(3, n-r, n) = 0$.
The diophantine equation
\[
3(n-r) n - (n-r)^2 - n^2 +1 = 0
\]
can be solved by the method of \cite[Section 3.1]{kronecker}.
The solutions are of the form
\[
(n-r, n) = R^k(1, 3), \quad k \ge 0,
\]
where
\[
R \colon \ZZ^2 \to \ZZ^2 \qquad \text{is given by} \qquad R(m, n) = (n, 3n-m).
\]
Note that $K(3, 1, 3)^\st$ is non-empty
because any $3 \times 1$-matrix with linearly independent entries in $\CC^3$ is stable.
According to \cite{drezet_reine}, we have canonical isomorphisms
\[
\N(3, m , n) \isom \N(3, n, 3n-m).
\]
It follows that $K(3, n-r, n)^\st$ and, a fortiori, $U^\st$ are non-empty
for all pairs
\[
(n-r, n) = R^k(1, 3), \quad k \ge 0.
\]
Denote $U_0^\st = U_0 \cap U^\st$.
Let $H_0^\st \subset H_0$ be the open subset of pairs $(Z, C)$ such that $Z \in \zeta(U_0^\st)$.
We also define $W_0^\st \subset W_0$ by the condition that $\f_{12}$ be stable as a Kronecker module
and put $M_0^\st = \rho(W_0^\st)$. Clearly, $\eta^{-1}(M_0^\st) = H_0^\st$.

\begin{proposition}
\label{proposition_9}
The morphism $\eta \colon H_0 \to M_0$ is surjective.
Assume that either
\[
2 \le r < \frac{n(\sqrt{5} - 1)}{2}
\]
or
\[
(n-r, n) \in \{ R^k(1, 3) \mid \ k \ge 0\} = \{(1, 3), \ (3, 8), \ (8, 21), \ \ldots \ \}.
\]
Then the generic fibres of $\eta$ are isomorphic to $\PP^{r-1}$.
Thus, $\M(n+r, n) \times \PP^{r-1}$ is a rational variety,
so $\M(n+r, n)$ and $\M(n+r, r)$ are unirational.
\end{proposition}

\begin{proof}
Consider $\f \in W_0$, let $\F = \Coker(\f)$ and let $C$ be the support of $\F$.
For $j = 1, \ldots, r$ denote by
\[
\f_j \colon (r-1) \O(-2) \oplus (n-r)\O(-1) \lra n\O
\]
the morphism obtained by deleting column $j$ from the matrix representing $\f$.
Denote by $f_{ij}$, $1 \le i \le n$, the maximal minor of $\f_j$ obtained by deleting row $i$.
Given $a = (a_1, \ldots, a_r) \in \PP^{r-1}$ we choose a morphism
\[
\alpha \colon (r-1) \O(-2) \oplus (n-r) \O(-1) \lra r \O(-2) \oplus (n-r) \O(-1)
\]
such that, for $1 \le i \le r$, $a_i$ is the maximal minor of $\alpha$ obtained by deleting row $i$.
We claim that $\psi = \f \alpha$ belongs to $U_0$, that is, the maximal minors of $\psi$ have no common factor.
Indeed, there is
\[
g \in \Aut(r\O(-2) \oplus (n-r)\O(-1))
\]
whose restriction obtained by deleting the first column equals $\alpha$.
Thus, $(\f g)_1 = \psi$, so any common factor of the maximal minors of $\psi$ divides $\det(\f g)$.
But the latter is irreducible, by hypothesis. Let $Z(a) = \zeta(\psi) \in \Hilb^0(l)$.
Note that $Z(a)$ is given by the ideal
\[
(a_1 f_{i1} + \cdots + a_r f_{ir}, \ 1 \le i \le n),
\]
so it does not depend on the choice of $\alpha$.
Clearly, $Z(a)$ is a subscheme of $C$ and $\eta(Z, C) = [\F]$,
showing that $\eta$ is surjective.
Consider the morphism
\[
\theta \colon \PP^{r-1} \to \eta^{-1}[\F], \qquad \theta(a) = (Z(a), C).
\]
To prove that $\theta$ is surjective, assume that $\eta(Z, C) = [\F]$.
Choose $\psi' \in \zeta^{-1}(Z)$.
We saw above that there is $\f' \in W_0$ such that $\f'_1 = \psi'$ and $\Coker(\f') \isom \F$.
There is $(g, h) \in G_W$ such that $\f' = h \f g$. Thus, $h^{-1} \psi' = (\f g)_1$, so $Z = \zeta((\f g)_1)$.

We will prove that $\theta$ is injective if $\F$ gives a point in $M_0^\st$,
which, by hypothesis, is non-empty.
Actually, the argument only uses the fact that $\f_{12}$ has trivial isotropy group, cf. \cite[Lemma 8.1]{pacific}.
Assume that $Z(a) = Z(b)$.
There are $\alpha$ and $\beta$ in $U_0$ corresponding to $a$, respectively, $b$.
There is $(g, h) \in G_U$ such that
$h \f \alpha g = \f \beta$. From the relation $h \f_{12} \alpha_{22} g_{22} = \f_{12} \beta_{22}$
and the fact that $\f_{12}$ has trivial isotropy group
we deduce that $h$ and $\alpha_{22}^{} g_{22}^{} \beta_{22}^{-1}$
are of the form $c$ times the identity, for some $c \in \CC^*$.
We may replace $\alpha$ with $c \alpha g$, and we may write $\f \alpha = \f \beta$.
Choose
\[
g' \in \Aut(r\O(-2) \oplus (n-r)\O(-1))
\]
such that the first row of $\alpha' = g' \alpha$ is zero
and put $\f' = \f {g'}^{-1}$, $\beta' = g' \beta$.
The maximal minors of $\alpha'$ and $\beta'$ are
$a_1', \ldots, a_r', 0, \ldots, 0$, respectively, $b_1', \ldots, b_r', 0, \ldots, 0$.
There is an automorphism $f$ of $\PP^{r-1}$ such that
\[
a' = (a_1', \ldots, a_r') = f(a), \qquad b' = (b_1', \ldots, b_r') = f(b).
\]
From the relation $\f ' \alpha' = \f' \beta'$ we see that the first row
of $\beta'$ is zero. If this were not the case, then the first column of $\f'$ would be a linear combination
of the remaining columns of $\f'$, which is absurd.
We get $b' = (1, 0, \ldots, 0) = a'$, hence $a = b$.

We conclude that $\theta \colon \PP^{r-1} \to \eta^{-1}[\F]$ is an isomorphism for
$[\F]$ belonging to the open subset of $M_0^\st$ of points having smooth fibre.
Clearly,
\[
\M(n+r, n) \times \PP^{r-1} \quad \text{is birational to} \quad \Hilb(l, n+r).
\]
The latter is rational because the map $\Hilb(l, n+r) \to \Hilb(l)$ is a projective bundle over $\Hilb^0(l)$
(with fibre of dimension $3n + 2r$).
Thus $\M(n+r, n)$ is unirational.
In view of the isomorphism
\[
\M(n+r, n) \isom \M(n+r, r)
\]
of \cite{rendiconti}, also $\M(n+r, r)$ is unirational.
\end{proof}

\end{document}